\documentclass[Svgc]{Svgc}
\usepackage{graphics}
\journalname{Graphs and Combinatorics}
\usepackage{fullpage}
\usepackage{amssymb,amsmath,amsfonts}
\usepackage{array}
\usepackage{amsmath}
\usepackage{amssymb}

\begin{document}

\newcommand{\bG}{{\bf G}}
\newcommand{\bB}{{\bf B}}
\newcommand{\bR}{{\bf R}}

\newcommand{\AP}{{\rm AP}}
\newcommand{\sr}{{\rm sr}}

\newtheorem{tha}{Theorem}
\renewcommand{\thetha}{\Alph{tha}}
\newtheorem{construction}[theorem]{Construction}
\newtheorem{observation}{Observation}
\newtheorem{fact}[theorem]{Fact}
\setlength{\textwidth}{17cm} \setlength{\oddsidemargin}{-0.1 in}
\setlength{\evensidemargin}{-0.1 in} \setlength{\topmargin}{0.0
in}

\def \df {\noindent {\bf Definition. }}

\newcommand{\dist}{{\rm dist}}


\def\qed{\hskip 1.3em\hfill\rule{6pt}{6pt} \vskip 20pt}

\linespread{1.0}
\input epsf
\def\epsfsize#1#2{1.0#1\relax}
\def\O{\text{O}}
\def\o{\text{o}}
\def\ex{\text{ex}}
\def\Z{\mathbb Z}
\def \lb{\lceil 8(n-1)/17 \rceil+1}
\def \bg{\lceil n/17 \rceil}

\title{Sub-Ramsey numbers for arithmetic progressions}
\author{Maria Axenovich and  Ryan Martin}

\institute{Department of Mathematics, Iowa State University, Ames,
IA, 50011. \\ {\it axenovic@math.iastate.edu,
rymartin@iastate.edu}}
\date{\today}
\maketitle

\begin{abstract}
Let the integers $1,\ldots,n$ be assigned colors. Szemer\'edi's
theorem implies that if there is a dense color class then there is
an  arithmetic progression of length three in that color. We study
the conditions on the color classes forcing totally multicolored
arithmetic progressions of length 3.

Let $f(n)$ be the smallest integer $k$ such that there is a coloring of
$\{1, \ldots, n\}$  without totally multicolored arithmetic progressions of length three and
such that each color appears on at most $k$ integers.
We  provide an exact value for $f(n)$ when $n$ is sufficiently large,
and all extremal colorings.
In particular, we show that $f(n)= 8n/17 + O(1)$.
This completely  answers a question of Alon, Caro and Tuza.
\end{abstract}


\finalreceive{December 21, 2005}

\section{Introduction}
In this paper we investigate colorings of sets of natural numbers.
We say that a subset is {\it monochromatic} if all of its elements
have the same color and we say that it is {\it rainbow} if all of
its elements have distinct colors. A famous result of van der
Waerden \cite{W} can be reformulated in the following way.

\begin{theorem}
For each pair of positive integers $k$ and $r$ there exists a
positive integer $M$ such that  any coloring of integers
$1,\ldots,M$ with  $r$ colors yields a monochromatic arithmetic
progression of length $k$.
\end{theorem}

This theorem was generalized by the following very strong statement of
Szemer\'edi \cite{S}.

\begin{theorem}
For every natural number $k$ and positive real number $\delta$
there exists a natural number $M$ such that every subset of
$\{1,\ldots,M\}$ of cardinality at least $\delta M$ contains an
arithmetic progression of length $k$.
\end{theorem}

This means that ``large'' color classes force monochromatic arithmetic progressions.
In this paper we invesigate conditions on the color classes which force a
totally multicolored arithmetic progression of length three.

Assume that the integers in  $\{1,\ldots,n\}$ are colored by  $r$ colors. Can we always find an
arithmetic progression of length $k$ so that all of its elements are
colored with distinct colors? We call such colored arithmetic
progressions {\it rainbow} $\AP(k)$.

The answer to this question is ``No'', for $r\leq
\lfloor\log_3n+1\rfloor $. The following coloring $c$ of $\{1,
\ldots, n\}$, given by Jungi\'c, et al.~\cite{J}, demonstrates
this fact. Let $c(i)=\max \{ q: i\mbox{ is divisible by }3^q\}.$
This coloring has no rainbow arithmetic progressions of length  $3$ or more.

It is an open question to determine certain  conditions which force
the existence of rainbow arithmetic progressions. There are two
natural approaches which can be studied. First, one can fix the
number of colors and require that each color class is not ``too
small''. Second, one can require that each color class is not
``too big'' to guarantee some rainbow arithmetic progression.

The first  approach  for $\AP(3)$ and three colors, among others,
was studied in \cite{J} and completely resolved by Fon-Der-Flaass
and the first author as follows.
\begin{theorem}[\cite{AF}] \label{AF}
Let $[n]$ be colored in three colors, each color class has size larger than
$(n+4)/6$. Then there is a rainbow $\AP(3)$. Moreover, for each $n=6k-4$  there
is a coloring of $[n]$ in three colors with the smallest color
class of size $k$ and with  no rainbow
$\AP(3)$.
\end{theorem}

The second approach was introduced and developed by Alon, et
al.~\cite{ACT}. It was called ``Sub-Ramsey numbers for arithmetic
progressions'' as a way to investigate the problem provided that
the size of the largest color class is bounded. Specifically, a
coloring of $[n]$ was called a {\bf sub-$k$-coloring} if every color
appears on at most $k$ integers. For a given $m$ and a given $k$, the {\it
Sub-Ramsey number}, $\sr(m,k)$, is defined to be the minimum $n_0$
such that any sub-$k$-coloring of $[n]$, $n>n_0$ contains a
rainbow $\AP(m)$.  When $m=3$, i.e., when the desired rainbow
arithmetic progressions are of size three, the following bounds
were proved in \cite{ACT}.
\begin{theorem}
As $k$ grows, $2k<\sr(3,k)\leq(4.5+o(1))k$.
\end{theorem}
In that paper it was suggested that the lower bound is close to
the correct order of magnitude for $\sr(3,k)$. Here, we show that
the truth is away from both the lower and upper bounds. In
theorem~\ref{main}, we compute tight bounds for $\sr(3,k)$ in a
dual form.  In particular, theorem~\ref{main} implies the
following:
\begin{theorem}
For any $k\geq 1$, $(17/8)k-4\leq\sr(3,k)\leq (17/8)k +10$.
\label{srmain}
\end{theorem}
Moreover, for $k$ large enough, we determine the value of
$\sr(3,k)$  exactly.

\section{Main Results}

\begin{definition}
We define $f(n)$ to be the smallest
integer $k$ such that there is a coloring of $[n]$ with the
largest color class of size $k$
and with no rainbow $\AP(3)$.
\end{definition}

The following proposition allows us to determine $\sr(3,k)$ from
$f(n)$:
\begin{proposition}
   The value $\sr(3,k)$ is the largest value of $n$ such that
   $k\geq f(n)$. \label{fandsr}
\end{proposition}

\begin{proof}
   Since there exists a $k$-bounded coloring of $[\sr(3,k)]$
   with no rainbow $\AP(3)$, $f(\sr(3,k))\leq k$.  Assume
   that $f(\sr(3,k)+1)\leq k$, then there is a
   $k$-bounded coloring of $[\sr(3,k)+1]$ with no rainbow
   $\AP(3)$, a contradiction.
\end{proof}

For the rest of the paper, we analyze the function $f(n)$.
Theorem~\ref{main} immediately implies the conclusion we draw in
theorem~\ref{srmain}.

We find an extremal coloring $c_0$ with no rainbow $\AP(3)$ and
with largest color class of the
smallest possible size. \\

\noindent
{\bf Construction.}
$$ c_0(i)=\begin{cases}
           \bG, & {\rm if} \quad i\equiv 0\pmod{17}, \\
           \bR, & {\rm if} \quad i \equiv  \pm 1,\pm 2,\pm 4,\pm
           8\pmod{17},
           \\
           \bB, & {\rm if} \quad i \equiv \pm 3,\pm 5,\pm 6,\pm
           7\pmod{17}.
        \end{cases} $$
\noindent
Let $q(I)$ be the size of the largest color class of $c_0$ in the interval $I$ and
$$Q(n) = \min \{q(I): \mbox{$I$ has length $n$}\}.$$
It can be easily verified  that $Q(n)= \lceil 8(n-1)/17 \rceil + \epsilon$, where
$\epsilon =
\begin{cases}
1, & n \equiv  3,5 \pmod {17},\\
0, & {\rm otherwise}.
\end{cases}
$

\begin{theorem} \label{main}
Let $n_0= 2600$. If   $n \geq  n_0$ then
$$ f(n)= Q(n).$$
Any extremal coloring of $\{1, \ldots, n\}$ is
colored identically to a subinterval of $\mathbb{Z}$ colored by $c_0$.
Moreover, for any $n\geq 1$,
$$Q(n) - 4 \leq f(n)\leq Q(n).$$
\end{theorem}

\begin{corollary}
$$  \left \lceil \frac {8(n-1)}{17}\right   \rceil  \leq f(n) \leq \left \lceil \frac {8(n-1)}{17} \right \rceil +1,$$ for $n\geq 2600$.
Moreover
$$  \frac {8(n-1)}{17} - 4   \leq f(n) \leq  \frac {8(n-1)}{17}+2,$$ for $n\geq 1$.
\end{corollary}

\begin{remark}
We did not try to optimize the constant $n_0$. A more careful
analysis of the proof results in a smaller number. We believe that
in fact $f(n)=Q(n)$ for all values of $n$ and this must be a
coloring of some subinterval of $\mathbb{Z}$ for all but a very
small number of values of $n$.
\end{remark}

\section{Definitions and Notations, Outline of the proof}
Let $[n]=\{1, \ldots, n\}$. For convenience, sometimes we shall
use the closed interval notation  $[1,n]$ for $[n]$. Let $c: [n]
\rightarrow \{\bR, \bG, \bB\}$. We say that a color $X\in \{\bR,
\bG, \bB\}$ is {\bf solitary} if there is no $x\in [n-1]$ such
that $c(x)=c(x+1)=X$. For a set $S\subseteq [n]$, we denote by ${
r(S), g(S), b(S)}$ the number of elements in $S$ colored $\bR,
\bG, \bB$ respectively. We write $|\bR|= r([n])$,  $|\bG|=g([n])$,
$|\bB|=b([n])$. If all elements of $S$ have the same color $X$, we
write $c(S)=X$.

We say that the interval $[x, x+i]$ is {\bf $X$-$X$-interval} if
$c(x)=c(x+i+1)=X$ and $c(x+j)\neq X$ for all $1\leq j\leq i$, note
that the left $X$ is included in the interval but the right one is
not. For a color $X$, we define a set $N(X)$ of {\bf  neighbors of
X} as follows $N(X) = \{ i\in [n]:  c(i+1)=X \quad {\rm or} \quad
c(i-1)=X\}$. For a sequence of colors $A_0, A_1, \ldots, A_k$, ~\
$A_i\in \{\bR, \bG, \bB\}$, we say that a coloring $c$ {\bf
contains $A_0A_1 \cdots A_k$} in the interval $I$ if there is an
integer $x\in I$, such that $x+k\in I$ and $c(x+i)=A_i$, $i=0,
\ldots, k$. Sometimes we shall simply say that $I$ contains
$A_0A_1 \cdots A_k$. We use subintervals of $[1,n]$ or subsets of
$[n]$ wherever convenient.

In order to prove our upper  bound on $f(n)$, we consider an arbitrary coloring of $[n]$ with
no rainbow $\AP(3)$ and first reduce the  analysis to the case of three colors  only.
We show that there must be a solitary color, say $\bG$. Moreover we show that
each number in the neighbor set of $\bG$ must have the same color, say $\bR$.
I.e., each integer colored $\bG$ is surrounded by two integers colored $\bR$.
Therefore the interval $[1,n]$ can be split into $\bG$-$\bG$ intervals and perhaps
some initial and terminal intervals containing no  $\bG$.
Next, we show that either each $\bG$-$\bG$ interval has many integers colored $\bR$, thus
arriving at a conclusion that $|\bR|\geq Q(n)$  or  that there are not too many integers colored $\bG$ and
either $|\bR|$ or $|\bB|$  is at  least $(n-|\bG|)/2 \geq Q(n)$.

We present the proof in the section \ref{main_proof}, and all
necessary technical lemmas  in sections \ref{lemmas_general},
\ref{lemmas_specific}.

\section{Proof of Theorem \ref{main}}   \label{main_proof}

Let $c$ be a coloring of $[n]$ with no  rainbow $\AP(3)$. We shall
conclude that one of the color classes has size at least $Q(n)$.
By lemma  \ref{merging}, we can assume that $c$ uses three colors,
say $\bR, \bG, \bB$. Lemma \ref{solitary} implies an existence of
a solitary color, without loss of generality $\bG$. If there are
only two numbers of color $\bG$, then either $\bR$ or $\bB$ has
size at least $(n-2)/2\geq 8(n-1)/17 + 3> Q(n)$, for $n\geq n_0$
and $(n-2)/2\geq Q(n)-3$ for $n\geq 1$. Otherwise, by lemma
\ref{neighbor_solitary}, we can assume that the neighbor set of
$\bG$ is colored $\bR$. We can also assume that there are two
consecutive numbers colored $\bB$ in $[n]$; otherwise, the
cardinality of $\bR$ is at least $(n-2)/2> Q(n)$, for $n\geq n_0$
and $(n-2)/2>Q(n)-3$ for $n\geq 1$.

Since $\bG$ is a solitary color and $\bR$ is the color of its
neighborhood, we see that $c$ looks as follows:
$$**\cdots ** \bR\bG\bR ** \cdots  ** \bR\bG\bR **\cdots ** \bR\bG\bR ** \cdots **  \bR\bG\bR ** \cdots ** ,$$
where $*\in \{\bR, \bB\}$. Furthermore, there is a $\bB\bB$
somewhere in $[n]$.\\~\\

\noindent
{\bf CASE  1.} All  $\bG$-$\bG$-intervals contain  $\bB\bB$.\\
Lemma \ref{up-to-15pattern} proves that the smallest length of a
$\bG$-$\bG$ interval containing $\bB\bB$ is $15$ and there is no
such interval of length $16$. Assume first that there is such an
interval of length $15$. Then lemma \ref{repeated_pattern} shows
that this  coloring  must be very specific, in particular, it is
defined up to translation on all integers except, perhaps, every
$15^{\rm th}$ one. So, in that case, lemma  \ref{repeated_pattern}
gives that $|\bR|\geq  8(n-1)/15-1 \geq 8(n-1)/17 +3> Q(n)$ for
$n\geq n_0$ and $8(n-1)/15-1>Q(n)-3$ for $n\geq 1$. If the
smallest $\bG$-$\bG$ interval has length $17$ then lemma \ref{17}
says that the coloring of $[n]$ must be a translation of $c_0$ for
all integers except, perhaps, every $17^{\rm th}$ one. In this
case $|\bR|= Q(n)$. Finally, if all intervals have length at least
$18$, lemma \ref{up-to-15pattern} proves that in fact, the
smallest interval has length $21$. Then $|\bG|\leq n/21 +1$. Thus
either $|\bB|$ or $|\bR|$ is at least $(n- |\bG|)/2 \geq (10n -
11)/21>Q(n)$ for $n\geq n_0$ and $(10n - 11)/21>Q(n)-3$ for $n\geq
1$.
~\\~\\

\noindent {\bf CASE 2.} There is a  $\bG$-$\bG$-interval
containing  no $\bB\bB$. \\
We split interval $[1,n]$ and find a lower bound on the number of
integers colored $\bR$ in each of those subintervals. There are
two subcases we shall treat. In case $2.1$, the initial
subinterval contains at least three $\bG$s, and we use our
structural lemmas. Otherwise, we have case $2.2$, in  which  we
apply case $1$ to a special subinterval. We shall define the
following special subintervals.\\
\indent $\circ$ $I_1$ is  the
longest initial segment of $[n]$ containing no  $\bB\bB$ and
ending with $\bG$, $I_1=[1,l]$,\\
\indent $\circ$
$I_2$ is  an interval following $I_1$, containing no  $\bB\bB$ except for the last two positions
which are colored $\bB\bB$,\\
\indent $\circ$
$I_3= [n]-I_1-I_2$,\\
\indent $\circ$
$I_0\subseteq I_1$ is the longest initial segment of $[1,n]$ containing no $\bG$,\\
\indent $\circ$
$I_2' = [l+1, 2l-1]$, $I_2''= I_2\setminus I_2'$.\\
\indent $\circ$
$I_t$ is the longest terminal subinterval of $[n]$ containing no $\bB\bB$.\\

{\bf Case 2.1} Let $g(I_1) \geq 3$. Let $g_i$ be the number of
$\bG$-$\bG$ intervals of length $i$ in $I_1\setminus \{l\}$. Lemma
\ref{GRG}(b) and \ref{GRG}(c) claims that there is no $\bG\bR\bG$
or $\bG\bR\bR\bG$ in $[n]$. Thus each $\bG$-$\bG$ interval in
$[n]$ has  length at least $4$ and $g_i=0$ for $i\leq 3$. In
particular, $|I_1| = |I_0| + \sum _{i\geq 4} i g_{i}+1$.

Since  $I_1$ contains no  $\bB\bB$ we have
\begin{equation}\label{r(I1)}
r(I_1)\geq  |I_0|/2 + \sum _{i\geq 4} ( i/2 +1)  g_{i}.
\end{equation}

Lemma \ref{I_1} states  that  $I_2' \subseteq I_2$ and $r(I_2')\geq r(I_1)$.
Since  $I_2''$ does not contain any $\bB\bB$  except at the last two positions,
$r(I_2'') \geq  |I_2''|/2 - 1$.
Thus
\begin{equation}\label{r(I2)}
r(I_2) =r(I_2')+r(I_2'') \geq r(I_1) + |I_2''|/2 -1.
\end{equation}
Finally, by lemma \ref{I_3},
\begin{equation} \label{r(I3)}
r(I_3)\geq (|I_3|-3)/4.
\end{equation}

We can summarize (\ref{r(I1)}), (\ref{r(I2)}) and  (\ref{r(I3)}) as follows.

\begin{eqnarray} \nonumber
|\bR| & = & r(I_1)+r(I_2')+r(I_2'') +r(I_3) \nonumber \\
& \geq & 2r(I_1)+ |I_2''|/2  -1  + (|I_3|-3)/4 \nonumber\\
& = & 2r(I_1) + |I_2''|/2 - 1 + (n - |I_1| - |I_2'|-|I_2''|-3)/4 \nonumber\\
& \geq & (n-3)/4 -1 + 2r(I_1) - |I_1|/2  \nonumber\\
& \geq & (n-7)/4 + 2\left[|I_0|/2+ \sum _{i\geq 4} g_{i} (i/2+1)\right] - (1/2) \left[ |I_0|+\sum_{i\geq 4} i g_{i}+1\right] \nonumber\\
& \geq & (n-9)/4 + |I_0|/2+ \sum_{i\geq 4} g_{i}(i/2+2)\nonumber \\
& \geq & (n-9)/4 + 4\sum_{i\geq 4} g_{i} \nonumber \\
& = & (n-9)/4 + 4 (g(I_1)-1).\nonumber
\end{eqnarray}
\noindent
Lemma \ref{I_3} implies  that $g(I_1)-1 = |\bG|-1-g(I_2\cup I_3)\geq |\bG|-3.$
So,
$$ |\bR|\geq (n-9)/4+4(|\bG|-3) . $$

Let $M=\max\{|\bR|,|\bB|\}$.  By definition, it is the case that
$|\bR|\leq M$ and $|\bG|\geq n-2M$.  As a result,
$$M\geq |\bR|  \geq  (n-9)/4 + 4(|\bG|-3)\geq (n-9)/4+4(n-2M-3).$$
Thus
$$ M \geq  \frac{17n-57}{36} \geq 8(n-1)/17+3\geq Q(n), $$
for  $n\geq n_0$.
We also have that $M \geq (17n-57)/36 \geq 8(n-1)/17>Q(n)-3$ for
all values of
$n\geq 1$. \footnote{Note that this is the only time we need the value of $2600$  for $n_0$, in all other calculations, a smaller bound of $900$ is sufficient.} \\

{\bf Case 2.2}
Let $g(I_1)\leq 2$.  By symmetry, we can also assume that $g(I_t)\leq 2$,
otherwise we can apply the previous calculation to the coloring defined as
$c'(i)  = c(n+1-i)$, $i\in [n]$.
Let $J= [n]\setminus (I_1 \cup I_t)$.
If $J$ contains no $\bG$ then $g([n])\leq 4$ and either $|\bR|$ or $|\bB|$
is at least $(n-4)/2 \geq 8(n-1)/17 + 3\geq  Q(n)$  for $n\geq n_0$,
moreover $(n-4)/2 \geq Q(n)-3$ for all $n\geq 1$.

If there is at least one  $\bG$ in $J$ then we  conclude that all
$\bG$-$\bG$ intervals in $J\cup \{l\}$ contain $\bB\bB$ by lemma
\ref{BB} and  that  $r(I_1)\geq |I_1|/2$ and $r(I_t)\geq |I_t|/2$.
As in case 1, we observe that if $J$ contains a $\bG$-$\bG$
interval of length $15$ then $|R| \geq 8(n-1)/15-1 \geq 8(n-1)/17
+3 \geq Q(n)$, for $n\geq n_0$.
In addition, if $J\cup \{l\}$ contains a $\bG$-$\bG$ interval of length $17$ then
lemma \ref{17} gives  that the coloring must be a translation  of $c_0$ except, perhaps on every
$17^{\rm th}$ position.  In this case, $|\bR| \geq Q(n)$.
Otherwise, the length of each $\bG$-$\bG$ interval is at least $21$.  This  follows from lemma \ref{up-to-15pattern}.
In that case,
$g(J)\leq  |J|/21+1 $. Thus $ |\bG| \leq g(J) + 4 \leq  (n-4)/21 + 5$.
Therefore either  $|\bR|$ or $|\bB|$ is at least $ (n - |\bG|)/2\geq  (10n-51)/21  \geq 8(n-1)/17+3>  Q(n)$,
for $n\geq n_0$,
moreover $|\bR|  \geq  8(n-1)/17 -1\geq  Q(n)-4$ for all $n\geq 1$.

This case concludes the proof of the theorem.
\qed

\section{General lemmas for colorings with no rainbow $\AP(3)$s}\label{lemmas_general}

\begin{lemma}\label{c_0}
The coloring $c_0$ does not have any rainbow $\AP(3)$s.
\end{lemma}

\begin{proof}
Consider  $\AP(3)$ at positions $i<j<k$ with $c(j)=\bG$. Then
$j=0\pmod{17} $ and  then $i=-k\pmod {17}$. Therefore, by
construction, $c(i)=c(k)$ and this $\AP(3)$ is not rainbow.

Now, let us have $\AP(3)$ at positions $i<j<k$ such that
$c(i)=\bG$. Then, since $i=0\pmod {17}$ we have $k=2j\pmod {17}$.
We claim that $c(j)=c(k)$ in this case simply by multiplying the
numbers in corresponding congruence classes by two as follows:

\begin{center}
\begin{tabular}{|l||r|r|r|r|r|r|r|r|} \hline
$x$ & 1 & 2 & 4 & 8 & 3 & 5 & 6 & 7 \\ \hline
$2x\pmod{17}$ & 2 & 4 & 8 & -1 & 6 & -7 & -5 & -3 \\ \hline
\end{tabular}
\end{center}

Therefore, in this case we see that this $\AP(3)$ is not rainbow
and  there is no rainbow $\AP(3)$ in our coloring.
\end{proof}

\begin{lemma}\label{merging}
Let $c$ be a coloring of $[n]$ with no rainbow $\AP(3)$, $n\geq
21$ and every color class of size at most $m$, $ (n+4)/6 \leq m <
(n-4)/2$. Then there is a coloring $c'$ of $[n]$ with no rainbow
$\AP(3)$, in three colors with each color class being the union of
some color classes of $c$ and such  each color class of $c'$ has
size at most $m$.
\end{lemma}

\begin{proof}
Let $A_1, A_2, \ldots$ be the color classes of $c$. Note first
that  if $c'$ is formed by merging color classes of $c$ then $c'$
does not have  rainbow $\AP(3)$s. If there were a rainbow $\AP(3)$
in $c'$, then it  must be a rainbow $\AP(3)$ in $c$, a
contradiction.

Assume first that there are two color classes $A_1$ and $A_2$ of
sizes more than  $(n+4)/6$. Consider $S=[n]-A_1-A_2$. Let the
color classes of $c'$  be $A_1,A_2,S$. If $|S|> (n+4)/6$ then the
new color are all of sizes at least $(n+4)/6$, thus there is a
rainbow $\AP(3)$ in $c'$ by Theorem \ref{AF}, a contradiction.
Otherwise, $|S|\leq (n+4)/6\leq m$ and  all  color classes in $c'$ have sizes at most $m$.

Now, assume that there is exactly  one color class of size more than  $(n+4)/6$, say $A_1$.
Let $T= A_2\cup A_3 \cup \cdots \cup A_q$ such that $|T|> (n+4)/6$ but
$|T \setminus A_q|\leq (n+4)/6$.  Then, we see that $|T|\leq (n+4)/3$. Therefore,
$n-|T|-|A_1|\geq n-(n+4)/3-m> (n+4)/6$. If we  make the
new color classes $A_1,T,[n]\setminus (T\cup A_1)$, then by Theorem \ref{AF},
there is a rainbow $\AP(3)$ in $c'$, a contradiction.

Finally, if each color class has cardinality less than $(n+4)/6$
then we choose color classes of $c'$ greedily. Let $B_1= A_1\cup
A_2\cup \cdots A_q$ and $B_2= A_{q+1} \cup \cdots A_r$ be two new
color classes such that $(n+4)/6< |B_i|\leq  (n+4)/3$, $i=1,2$.
Let $B_3=[n]\setminus (B_1\cup B_2)$. Then $|B_3|\geq (n-8)/3$. If
$n\geq 21$, then $|B_3|> (n+4)/6$ and we again apply Theorem
\ref{AF} to get a rainbow $\AP(3)$ in $c'$ a contradiction.
\end{proof}

\begin{lemma}[\cite{AF}]
Let $c$ be a coloring of $[n]$ in three colors  with no rainbow
$\AP(3)$. Let there be integers $x$ and $z$,  $1\leq x<z<n$ such
that $c(x)=c(x+1)=X$ and $c(z)=c(z+1)=Z$, $X\neq Z$. Then there is
$w$, $x<w<z$ such that ($c(w)=X$, $c(w+1)=Z$) or
 ($c(w)=Z$, $c(w+1)=X$). \label{BBRR}
\end{lemma}

\begin{lemma} \label{solitary}
Let $c$ be a coloring of $[n]$ in three colors with no rainbow
$\AP(3)$.  Then there is a solitary color.
\end{lemma}

\begin{proof}
Assume the opposite. Let $c$ be a coloring of $[n]$ with colors
\bR, \bG, \bB~and such that each color appears on consecutive positions
somewhere in $[n]$. In particular, there are numbers $1\leq
x<y<z<n$ such that, without loss of generality, $c(x)=c(x+1)=\bR$,
$c(y)=c(y+1)=\bG$, and $c(z)=c(z+1)=\bB$, and such that
there are no two consecutive integers colored $\bB\bB$ or $\bR\bR$ in the interval
$[x+1, z]$.

By lemma~\ref{BBRR}, there is a $w$, with $x<w<z$, such that
($c(w)=\bR$ and $c(w+1)=\bB$) or ($c(w)=\bB$ and $c(w+1)=\bR$).
Assume without loss of generality that $x<w<y$ and that $w$ is
closest to $y$ satisfying this property, and $c(w)=\bR$,
$c(w+1)=\bB$. Note that $w+1<y-1$, otherwise $\{w,w+1,w+2\}$ will
be a rainbow $\AP(3)$. But now $c(w+2)=\bB$ otherwise we shall
contradict the choice of $w$. Therefore, we have
$c(w+1)=c(w+2)=\bB$, a contradiction.
\end{proof}

\begin{lemma} \label{neighbor_solitary}
Let $c$ be a coloring of $[n]$ in three colors \bR,\bG,\bB~with no
rainbow $\AP(3)$. Let color \bG ~be solitary.  Then, either the
neighbor set of \bG~is monochromatic or there are at most two
numbers $x,y$ with $c(x)=c(y)=\bG$.
\end{lemma}

\begin{proof}
Note first that if $c(x)=\bG$, for some $x\in\{2,\ldots,n-1\}$
then $c(x-1)=c(x+1)\in \{\bB,\bR\}$. Now, assume that there are
two integers $x,y$, $1\leq x<y\leq n$, such that $c(x)=c(y)=\bG$
but $c(z)\neq\bG$ for all $x<z<y$ and such that $c(x+1)=\bR$ and
$c(y-1)=\bB$. Assume that there are at least three integers
colored $\bG$. Then, it is easy to see that we may assume that
$x\geq 2$ or $y\leq n-2$. Let $y$  be at most $n-2$, without loss
of generality.

If $y+x$ is odd then $c((y+x+1)/2)=\bR$ and  $c((y+x+1)/2)=\bB$ which follows from
considering the $\AP(3)$ $\{x+1,(x+y+1)/2,y\}$ and $\{x,(x+y+1)/2,y+1\}$, respectively,  a contradiction.

If $y+x$ is even and $c(y+2)=\bB$, we have $c(x+2)=\bR$.  Then
$c((x+y+2)/2)=\bR$  and $c((x+y+2)/2)=\bB$  from the $\AP(3)$  $\{x+2,(x+y+2)/2,y\}$, and the $\AP(3)$ $\{x,(x+y+2)/2,y+2\}$,
 a contradiction.

If $y+x$ is even and $c(y+2)=\bG$, consider the largest $w$, $x<w<y$
such that $c(w)=c(w+1)=\bR$. Then one of $w+y$ and $w+1+y$ is even.
Assume, without loss of generality, that $w+y$ is even. Then
$(w+y)/2$ and $(w+y+2)/2$ will have to have color \bR~ because of
$\AP(3)$s $\{w,(w+y)/2,y\}$ and $\{w,(w+y+2)/2,y+2\}$, a
contradiction to maximality of $w$.
\end{proof}

\section{Lemmas specific to the main theorem}\label{lemmas_specific}

In all of the following lemmas we consider a coloring $c$ of $[n]$ in three colors
$\bR, \bG, \bB$ with a solitary color $\bG$ having all neighbors of color $\bR$.
We also assume that this coloring has two consecutive integers colored $\bB$.
The intervals $I_1, I_2, I_3$ are defined as in the proof of the theorem in section \ref{main_proof}.

\begin{lemma} \label{GRG}
~\\
\noindent
(a) If $x\in [1,n-1]$ and  $c(x), c(x+1) \in \{\bG, \bB\}$ then $c(x)=c(x+1)=\bB$.\\
(b) $[1,n]$ does not contain  $\bG\bR\bG$ \\
(c)  $[1,n]$  does not contain $\bG\bR\bR\bG$.\\
(d) If $x\in [1,n-2]$ and $c(x), c(x+2) \in \{\bG, \bB\}$ then $c(x)=c(x+2)=\bB$.\\
\end{lemma}

\begin{proof}
~

(a)  Note that having $c(x)=c(x+1)=\bG$ is impossible since $\bG$
is a solitary color. Having exactly one integer $x$ or $x+1$ of
color $\bG$ and another of color $\bB$ is impossible since the
neighbors of $\bG$ are colored with $\bR$.

(b)  Without loss of generality, we may assume that there are integers $w,y\in [n]$,
   $y>w$ and  such that  $w,w+1,w+2$ is
   colored \bG\bR\bG~ and $y$  is the least integer such that
   $c(y)=c(y+1)=\bB.$  If $y$ has the same parity as $w$ then  the $\AP(3)$
   $\{w, (w+y)/2, y\}$ and $\{w+2, (w+2+y)/2, y\}$ imply that
   $c((w+y)/2)=c((w+2+y)/2)=\bB$.
    If $y+1$ has the same parity as $w$ then  the $\AP(3)$
   $\{w, (w+y+1)/2, y+1\}$ and $\{w+2, (w+2+y+1)/2, y+1\}$ imply that
   $c((w+y+1)/2)=c((w+2+y+1)/2)=\bB$.
   This is a contradiction to the minimality of $y$.

(c)  Without loss of generality, we may assume that there are integers $y,w \in [n]$
such that  $w, w+1, w+2, w+3$ is colored $\bG\bR\bR\bG$
and that $y$ is the least integer such that $c(y)=c(y+1)=\bB$.
If $w+y$ is even, then consider the following $\AP(3)$s:
$\{w, (w+y)/2, y\}$ and $\{w+2, (w+2+y)/2, y\}$. It follows that
$c((w+y)/2)=c((w+y+2)/2)=\bB$. Since $y>(w+y)/2>w $, we have a contradiction to the minimality of $y$.
If $w+y$ is odd, the consider the following $\AP(3)$s:
$\{w, (w+y+1)/2, y+1\}$ and $\{w+3, (w+3+y)/2, y\}$. It follows that
$c((w+y+1)/2)=c((w+y+3)/2)=\bB$. Since $y>(w+y+1)/2>w $, we have
a contradiction to the minimality of $y$.

(d) Note that $c(x)=c(x+2)=\bG$ is impossible  because of b).
If $\{c(x), c(x+2)\}= \{\bB, \bG\}$ then, since
$c(x+1)=\bR$, $\{x,x+1,x+2\}$ is a rainbow $\AP(3)$.
\end{proof}

\begin{lemma} \label{BB}
Let $x<y$, $c(x)=c(y)=\bG$ and both intervals $[1,x]$ and $[y,n]$ contain  $\bB\bB$s.
Then $[x,y]$ contains $\bB\bB$.
\end{lemma}

\begin{proof}
   Let $w$ be the largest number such that $w<x$ and  $c(w)=c(w+1)=\bB$.
   Let $z$ be the smallest number such that $z>y$ and $c(z)=c(z-1)=\bB$.
 Assume without loss of generality that $x-w\leq z-y$.
   By considering the $\AP(3)$s $\{w,x,2x-w\}$ and
   $\{w+1,x,2x-w-1\}$, we have  that $c(2x-w-1), c(2x-w) \in \{\bB, \bG\}$,
and using lemma \ref{GRG} a), we have  $c(2x-w-1)=c(2x-w)=\bB$.
   If    $x<2x-w-1<y$, then we are done. Otherwise, $2x-w-1 >y$ and
   $2x-w-1 - y < z-y$, a contradiction to the choice of $z$.
\end{proof}

\begin{table}[h]\label{t}

$$ \begin{array}{c|rrrrrrrrrrrrrrrrrrrrr} \hline
    {\rm Interval} &&&&&&&&&&&&&&&&&&&&&\\
 {\rm Length} &\multicolumn{5}{l}{\rm  Coloring} &&&&&&&&&&&&&&&\\
\hline
      6   & \bG & \bR & \bR & \bB & \bR & \bR &  & & & &
      & & & & & & &  & & &\\ \hline
      9   & \bG & \bR & \bR & \bB & \bR & \bR & \bB & \bR &
      \bR &  & & & & & & & & & & & \\ \hline
      10  & \bG & \bR & \bR & \bR & \bR & \bB & \bR & \bR &
      \bR & \bR &  & & & & & & &  & & &\\ \hline
      12  & \bG & \bR & \bR & \bB & \bR & \bR & \bB & \bR &
      \bR & \bB & \bR & \bR &  & & & &  & & & \\ \hline
      14  & \bG & \bR & \bR & \bR & \bR & \bR & \bR & \bB &
      \bR & \bR & \bR & \bR & \bR & \bR & & & & & & &\\ \hline
      15 & \bG& \bR& \bR& x&\bR& y& x& \bR& \bR& x& y& \bR& x& \bR& \bR&&& & & &\\ \hline
      17 & \bG & \bR & \bR & \bB& \bR& \bB& \bB& \bB& \bR& \bR & \bB& \bB& \bB& \bR & \bB & \bR& \bR & & & &\\  \hline
18 & \bG & \bR & \bR & \bB& \bR& \bR& \bB& \bR& \bR& \bB & \bR& \bR& \bB& \bR & \bR & \bB& \bR & \bR&  & &\\  \hline
20  & \bG & \bR & \bR & \bR& \bR& \bB& \bR& \bR& \bR& \bR & \bB& \bR& \bR& \bR & \bR & \bB& \bR & \bR& \bR& \bR &\\  \hline
21 &\bG & \bR & \bR & x& \bR& \bR& x& y& \bR& x & \bR& \bR& x& \bR & y & x& \bR &\bR & x&\bR &\bR\\
\hline \end{array} $$
\caption{Colorings of $\bG$-$\bG$ intervals of lengths at most $21$ containing $\bB$. Here, $x,y\in \{\bR, \bB\}$. }
\end{table}

\begin{lemma} \label{up-to-15pattern}
Let $I$ be a $\bG$-$\bG$ interval with at least one $\bB$. Then
for each such   $I$ of length at most $21$, $I$ must be colored as
in table $1$.
\end{lemma}

\begin{proof}
Let $I=[0, k-1]$; i.e.,  $c(0)=c(k)=\bG$.  Because there is no
rainbow $\AP(3)$, we must have that $c(x)=c(2x)$ for all $x<k/2$
and $c(2x-k)=c(x)$ for all $x>n/2$.  Since the neighbor set of
\bG~is \bR, $c(1)=c(k-1)=\bR$. With these conditions we can
exhibit all possible colorings of $I$. The ones with at least one
$\bB$ are listed in table $1$, for $1 \leq k\leq 21$.
 \end{proof}

Now we present the main structural lemma.

\begin{lemma} \label{repeated_pattern}
Let $[x, x+14]$  be a $\bG$-$\bG$ interval containing $\bB\bB$. Then
\begin{equation} \nonumber
c(z)=
\begin{cases}
\bR, & {\rm if } \quad (z-x) \equiv  \pm 1,  \pm 2, \pm 4, \pm 7 \pmod {15},\\
\bB, & {\rm if} \quad (z-x) \equiv  \pm 3, \pm 5, \pm 6 \pmod {15}.
\end{cases}
\end{equation}
\end{lemma}

\begin{proof}
To simplify our calculations, we shift the indices so that considered $\bG$-$\bG$ interval
is $[0,14]$ and the whole interval being colored is $[1-x, n-x]$.
The lemma \ref{up-to-15pattern} shows that the coloring of
$[0,15]$ must be as follows:
   \begin{center}
      \bG\bR\bR\bB\bR\bB\bB\bR\bR\bB\bB\bR\bB\bR\bR\bG.
   \end{center}
In particular, we have that
\begin{equation}\label{i-2i}
c(2i)=c(i),\quad  c(2i-1)=c(7+i), \quad i\in\{1,\ldots, 7\}.
\end{equation}

Let $A=[-w+1,z-1]$ be the largest interval having  the coloring $c$ as in the statement of the lemma.
I.e., for each $y\in A$

\begin{equation}
 c(y)=
\begin{cases}
                                 \bR, & {\rm  if}\quad  y\equiv \pm 1,\pm 2,\pm 4,\pm 7\pmod{15}, \\
                                 \bB, & {\rm  if} \quad y\equiv  \pm 3,\pm 5,\pm 6\pmod{15}
\end{cases}\nonumber
\end{equation}

Let $ z= 15k + i$, $0< i<15$.  If $z\leq n-x$, we shall show that $z$ must be colored as in $c$, thus contradicting the
maximality of $[-w+1,z-1]$. By symmetry, it will be the case  that if $-w\geq  1-x$ then
$-w$ must be colored as in $c$, again contradicting the maximality of $[-w+1, z-1]$.
Therefore we shall conclude that $A= [-w+1, z-1]= [1-x, n-x]$.

First we show that $c(z)\neq \bG$ if $i\neq 0$. Assume that $c(z)=c(15k+i)=\bG$.
If $i\in \{4, 5, 6,$ $7,8,  10, 11,12, 13,14\}$
then either $c(z-1)=\bB$ or ($c(z-2)=\bB$ and $c(z-1)=\bR$).
We arrive at a contradiction since the neighbors of $\bG$ are colored
$\bR$ and we can not have three consecutive numbers colored $\bB\bR\bG$.
For $i\in \{1,2,3,9\}$ we consider the following $\AP(3)$s: $\{15k-3, 15k-1, 15k+1\}$,
 $\{15k-6, 15k-2, 15k+2\}$,    $\{15k-5, 15k-1, 15k+3\}$,  $\{15k+5, 15k+7, 15k+9\}$.
Note that the first two terms in each of these four $\AP(3)$s have distinct  colors from the set $\{\bR, \bB\}$,
thus the last terms can not be colored with $\bG$.

Next we show that $c(15k+i)=c(i)$.\\

{\bf Case 1.} $k$ is even, $i$ is even. \\
Consider  $\AP(3)$ $\{0, (15k+i)/2, 15k+i\}$. Since
 $c((15k+i)/2)= c(15(k/2) + i/2) = c(15(k/2) + i) = c(i),$
we have that $c(15k+i)=c(i)$.

{\bf Case 2.} $k$ is odd, $i$ is odd. \\
Consider  $\AP(3)$ $\{0, (15k+i)/2, 15k+i\}$. Since
 $c((15k+i)/2)= c(15((k-1)/2) + (15+i)/2) = c(15((k-1)/2) + 15+i) = c(i),$
we have that $c(15k+i)=c(i)$.

 {\bf Case 3.} $k$ is odd, $i$ is even. \\
Consider  $\AP(3)$ $\{15, (15(k+1)+i)/2, 15k+i\}$. Since
 $c((15(k+1)+i)/2)= c(15((k+1)/2) + i/2) = c(15((k+1)/2) +i) = c(i),$
we have that $c(15k+i)=c(i)$.

 {\bf Case 4.} $k$ is even, $i$ is odd. \\
Consider  $\AP(3)$ $\{15, (15k+i+15)/2, 15k+i\}$. Since
 $c((15k+i+15)/2)= c(15(k/2) + (i+15)/2) = c(15(k/2) +(i+15)) = c(i+15)=c(i),$
we have that $c(15k+i)=c(i)$.
\end{proof}

\begin{lemma} \label{17}
Let $[x, x+16]$  be a $\bG$-$\bG$ interval. Then
\begin{equation} \nonumber
c(z)=
\begin{cases}
\bR, & {\rm if } \quad (z-x) \equiv  \pm 1,  \pm 2, \pm 4, \pm 8 \pmod {17},\\
\bB, & {\rm if} \quad (z-x) \equiv \pm 3, \pm 5, \pm 6, \pm 7 \pmod {17}.
\end{cases}
\end{equation}
\end{lemma}

\noindent
{\bf Remark:} The proof is almost identical to the proof of the previous lemma
and can be easily mimicked by replacing $15$ with $17$  and modifying
corresponding indices.

\begin{lemma} \label{I_1}
$|I_1|\leq  |I_2|+1$ and $r(I_2')\geq r(I_1)$.
\end{lemma}

\begin{proof}
Assume first that $|I_1|\geq |I_2|+2$. Let $I_1 = [1, l]$ and
$I_2= [l+1, b+1]$. Recall that $c(l)=\bG$ and $c(b)=c(b+1)=\bB$.
The following $\AP(3)$s: $\{2l-b,  l,b\}$ and $\{2l-b-1, l, b+1\}$
and lemma \ref{GRG}(a) imply that $c(2l-b)=c(2l-b-1)=\bB$, a
contradiction to the fact that $I_1$ does not contain $\bB\bB$. To
prove the second statement, consider $\{x, l, 2l-x\}$, where $x\in
I_1$ and $c(x)=\bR$. Since  $2l-x \in I_2'$ and  $c(2l-x)\neq
\bG$, we have $c(2l-x)=\bR$. Therefore,
 for each $x\in I_1$ such that $c(x)=\bR$ there is
a unique $y\in I_2'$ such that $c(y) =\bR$.
\end{proof}

\begin{lemma} \label{I_3}
If $g(I_1)\geq 3$ then $g(I_2\cup I_3)\leq 2$ and  $r(I_3)\geq (|I_3|-3)/4$.
\end{lemma}

\begin{proof}

Assume that $I_2\cup I_3$ contains at least three integers colored $\bG$.
Since $g(I_2)=0$ by definition of $I_2$, we have $g(I_3)\geq 3$.
We know that $I_1$ contains at least three $\bG$s as well.
Then, there are $x,x'\in I_1$ and $y, y'\in I_3$, such that
$c(x)=c(x')=c(y)=c(y')=\bG$, $x$ and $x'$ are of the same parity and
$y$ and $y'$ are of the same parity.
Let $x<x'$ and $y<y'$.
Let $b$ be the smallest integer such that $c(b)=c(b+1)=\bB$.
Note that $x'<b<y$.

{\bf Claim:} $n<2b+2-x'$. Assume not, then $2b+2-x' \in [1,n]$,
thus considering $\AP(3)$s $\{x', b, 2b-x'\}$ and $\{x', b+1,
2b+2-x'\}$ we see that $c(2b-x')= c(2b+2-x')=\bB$. Now, the
$\AP(3)$s $\{x,  b- (x'-x)/2  , 2b-x'\}$ and $\{x, b+1- (x'-x)/2,
2b+2-x'\}$ show that $c(b-(x'-x)/2)= c(b+1-(x'-x)/2)=\bB$. This
contradiction to minimality of $b$ proves the claim.

Let $z$ be the largest number such that $z<y$ and
$c(z)=c(z+1)=\bB$. Observe that $2z-y  \geq   2b -y  \geq  n-2 +x'
-y +1  = n - (y-x') -1$. Since $x'\geq 4$ and $y\leq n$, we have
that $2z-y \geq n-n+4-1\geq 3$. Therefore we can consider the
following $\AP(3)s$: $\{2z-y, z, y\}, \{2z-y+2, z+1, y\}$, which
imply  that $c(2z-y)=c(2z-y+2)=\bB$. Then $\{2z-y, z+
(y'-y)/2,y'\}$, $\{2z-y+2, z+1 +(y'-y)/2, y'\}$ give us that $c(z
+ (y'-y)/2)= c(z+1 + (y'-y)/2) =\bB$, contradicting maximality of
$z$.

This proves that there are at most two integers colored $\bG$ in $I_3$. In order to prove the second
statement of the lemma we show that $I_3$ does not contain $\bB\bB\bB\bB$.

Assume that there is $y\in I_3$ such that $y+3\in I_3$ and
$y,y+1,y+2,y+3$ is colored $\bB\bB\bB\bB$.
Assume that $y$ and $y+2$ have  the same parity as $x'$ (otherwise take
$y+1$ and $y+3$).
Then $\{x', (y+x')/2, y\}$ and $\{x, (y+x')/2+1, y+2\}$ imply that
$c((y+x')/2)=c((y+x')/2+1)=\bB$.
Using the claim, we have that  $y\leq n-3 < 2b+2-x' -3$. Thus $ (y+x')/2 < (2b
-1)/2 < b$, a  contradiction to the minimality of $b$.
\end{proof}

{\bf Acknowledgments} We would like to thank Zsolt Tuza for
bringing this problem to our attention.  We also thank an
anonymous referee for helpful comments.

\end{document}